\documentclass[11pt]{article}

\usepackage{amsfonts,amssymb,amsmath,latexsym,xcolor,epsfig}

\usepackage{mathrsfs}

\usepackage{hyperref}
\usepackage{amsthm}
\usepackage{enumerate}

\newtheorem{theorem}{Theorem}[section]

\newtheorem{lemma}[theorem]{Lemma}
\newtheorem{corollary}[theorem]{Corollary}

\newtheorem{claim}{Claim}
\newtheorem*{claim*}{Claim}

\theoremstyle{definition}

\theoremstyle{remark}
\newtheorem*{remark*}{Remark}
\newtheorem*{notation*}{Notation}

\numberwithin{equation}{section}

\newcommand\eps{\varepsilon}
\newcommand\Z{\mathbb{Z}}
\newcommand\group{\Z/N\Z}

\title{\vspace{-0.7cm}A connection between matchings and removal in abelian groups}
\author{James Aaronson\thanks{Contact email:\texttt{ james.aaronson@maths.ox.ac.uk}}}

\date{}
\begin{document}
\maketitle

\begin{abstract}
In a finite abelian group $G$, define an additive matching to be a collection of triples $(x_i, y_i, z_i)$ such that $x_i + y_j + z_k = 0$ if and only if $i = j = k$. In the case that $G = \mathbb{F}_2^n$, Kleinberg, building on work of Croot-Lev-Pach and Ellenberg-Gijswijt, proved a polynomial upper bound on the size of an additive matching. Fox and Lov\'{a}sz used this to deduce polynomial bounds on Green's arithmetic removal lemma in $\mathbb{F}_2^n$. 

If $G$ is taken to be an arbitrary finite abelian group, the questions of bounding the size of an additive matching and giving bounds for Green's arithmetic removal lemma are much less well understood. In this note, we adapt the methods of Fox and Lov\'{a}sz to prove that, provided we can assume a sufficiently strong bound on the size of an additive matching in cyclic groups, a similar bound should hold in the case of removal.
\end{abstract}

\section{Introduction}\label{sec:1}

In an abelian group $G$, define a \emph{triangle} to be a triple of elements $x$, $y$ and $z$ with $x + y + z = 0$. Green's arithmetic triangle removal lemma \cite{green2005szemeredi} states that for any $\eps > 0$, there is a $\delta > 0$ such that the following holds. Whenever $X$, $Y$ and $Z$ are subsets of $G$ such that there are at most $\delta N^2$ triangles $x + y + z = 0$ with $x \in X$, $y \in Y$ and $z \in Z$, we can remove at most $\eps N$ elements from $X$, $Y$ and $Z$ to remove all of the triangles. The bounds in \cite{green2005szemeredi} are quite weak; $1/\delta$ is given as a tower of twos of height polynomial in $1/\eps$. The best known bounds for this problem in general are still of tower type.

In \cite{FLpaper}, Fox and Lov\'{a}sz proved a much stronger bound on $\delta$ in the case of $G = \mathbb{F}_p^n$ for a fixed prime $p$; namely, that $1/\delta$ is bounded by a polynomial in $1/\eps$.  Define an \emph{additive matching} to be a collection of triples $(x_i, y_i, z_i)$ such that $x_i + y_j + z_k = 0$ if and only if $i = j = k$. These are also called tricolored sum-free sets, and can be represented by $(X,Y,Z)$, where $X = \{x_i\}$, $Y = \{y_i\}$ and $Z = \{z_i\}$.  Building on the groundbreaking work on the cap set problem by Croot-Lev-Pach \cite{1605.01506} and subsequent work by Ellenberg-Gijswijt \cite{1605.09223}, Kleinberg \cite{1605.08416} gave a polynomial upper bound for the size of an additive matching in $G$ in the case that $G = \mathbb{F}_2^n$, and Blasiak-Church-Cohn-Grochow-Naslund-Sawin-Umans \cite{1605.06702} extended this to $\mathbb{F}_q^n$ for a fixed prime power $q$. The argument by Fox and Lov\'{a}sz made use of these results to prove the polynomial bounds on removal.

Polynomial bounds on removal are much stronger than could possibly hold in general groups. Indeed, using Behrend's construction \cite{MR0018694} of a large subset of $\group$ with no 3-term arithmetic progressions, it is possible \cite{MR519318} to show that the best one could hope for is

\begin{equation*}
\eps \ll \exp\left(-c \sqrt{\log(1/\delta)}\right).
\end{equation*}

\noindent The goal of this note is to adapt the arguments of Fox and Lov\'{a}sz to show that, in the context of cyclic groups, good bounds on additive matchings give good bounds on removal.

Assume that, in a cyclic group of order $M$,  the density of an additive matching is bounded above by $f(M)$ for some function $f$. Assume that $f(M)$ can be taken to be decreasing as $M$ increases, but that $Af(A) < Bf(B)$ for $A < B$; these conditions correspond to the claim that the maximum size of an additive matching increases as the size of the group increases, but the maximum density decreases. Observe that Behrend's example guarantees that 

\begin{equation*}
f(M) \gg \exp\left(-c \sqrt{\log(M)}\right)
\end{equation*}

\noindent because, if $A$ is a progression free set, then $(A, -2A, A)$ is an additive matching.

Suppose further that there exists a function $g$ such that $g(\rho)$ increases as $\rho$ decreases, $\sum_{i=1}^{\infty} \frac{1}{g(2^{-i})} \leq \frac{1}{4}$ and $g(\rho)^2 f\left(\frac{1}{g(\rho)\rho}\right)$ is decreasing as $\rho$ decreases for $\rho < \alpha$, for some absolute constant $\alpha$. $g$ plays the same role here as in \cite{FLpaper}.

We are now ready to state Theorem \ref{MainTheorem}.

\begin{theorem}\label{MainTheorem}
Suppose that $A$, $B$ and $C$ are subsets of $\group$ for some $N \in \mathbb{N}$ with the property that there are at most $\delta N^2$ triangles $a + b + c = 0$ with $a \in A$, $b \in B$ and $c \in C$.

Then, we can remove all of the triangles by deleting at most $\eps N$ elements from $A$, $B$ and $C$, where $\eps$ satisfies

\begin{equation}\label{MTEQ}
\eps \ll g(\delta) \sqrt{f\left(\frac{1}{g(\delta)\delta}\right)}.
\end{equation}

\end{theorem}

We can deduce some consequences of this:

\begin{corollary}\label{HappyCor}
Suppose that we have the best possible bound on the size of an additive matching, namely a Behrend-type bound. In particular, we can take $f(N)$ to be $\exp(-c\sqrt{\log{N}})$ for some constant $c$. Then $g(\rho) = k\log^2(1/\rho)$ suffices, and we deduce the bound

\begin{equation*}
\eps \ll \exp(-c_1\sqrt{\log 1/\delta})
\end{equation*}

\noindent for some other constant $c_1$.

\end{corollary}

\begin{corollary}\label{SadCor}
Suppose that the much more pessimistic bound $f(N) = \log^{-2 - \gamma} N$ holds, for some constant $\gamma > 0$. Then, we can take $g(\rho) = k\log^{1 + \gamma/3}(1/\rho)$, and we deduce that 

\begin{equation*}
\eps \ll \log({1/\delta})^{-O(1)}.
\end{equation*}

\end{corollary}

Observe that we cannot deduce anything nontrivial if the assumption on $f$ is weaker, because of the need for $\sum_{i=1}^{\infty} \frac{1}{g(2^{-i})}$ to converge.

%%%%%%%%%%%%%%%%%%%%
\begin{remark*}
A converse of sorts to Theorem \ref{MainTheorem}, namely that bounds on removal imply similar bounds on the maximal size of an additive matching, is relatively trivial. Indeed, suppose that, whenever subsets $X$, $Y$ and $Z$ of a cyclic group $G = \group$ define at most $\delta N^2$ triangles, the triangles can be removed by deleting at most $\eps N$ elements, where $\eps \ll f(1/\delta)$.

Then, an additive matching of size $\theta N$ defines at most $N = \frac{1}{N} N^2$ triangles, and requires removal of at least $\theta N$ elements to remove the triangles. Thus, $\theta \ll f(N)$, which can be seen to be the partial converse we wanted.

\end{remark*}
%%%%%%%%%%%%%%%%%%%%
Throughout this note, we will use the notation $x \ll y$ to mean that, for some absolute constant $C$ independent of any variables, $x \leq Cy$.

The author is supported by an EPSRC grant EP/N509711/1. The author a DPhil student at Oxford University, and is grateful to his supervisor, Ben Green, for his continued support.

\section{Theorem \ref{MainTheorem} for $N$ prime}

In this section, following the approach in \cite{FLpaper}, we we prove Theorem \ref{MainTheorem}, in the case that $N$ is prime. We start with a lemma, which is an analogue of Lemma 5 from \cite{FLpaper}:

\begin{lemma}\label{BigLem}
Suppose we have three subsets of $\group$, $X$, $Y$ and $Z$, with the property that, for each $x \in X$, there are between $\delta_1 N$ and $\delta_2 N$ elements $y \in Y$ such that $z = -x -y$ is in $Z$. Suppose that the same holds with the positions of $X$, $Y$ and $Z$ permuted.

Then, we deduce that $|X|$ satisfies

\begin{equation}\label{L3Conc}
|X| \ll \frac{\delta_2}{\delta_1} f(\delta_2^{-1})N,
\end{equation}

and similar inequalities for $|Y|$ and $|Z|$.

\end{lemma}

\begin{proof}
Choose $a$, $b$ and $d$ uniformly and independently from $\group$ such that $d$ is nonzero. Let $L = 2l + 1$ be an odd positive integer such that $1/20 \leq L\delta_2 \leq 1/10$. We may assume that $\delta_2$ is small enough that $L > 20$ by adjusting the implicit constant in \ref{L3Conc}, and so such a choice of $L$ exists.

We say that a triangle $x + y + z = 0$ is \emph{valid} if and only if $x \in I_X := a + [-l,l]d$ and $y \in I_Y := b + [-l,l]d$. Observe that this will imply that $z \in I_Z := -a-b + [-2l,2l]d$. We say that a valid triangle is \emph{good} provided that each of $x$, $y$ and $z$ is in only one valid triangle, namely the triangle in question.

\begin{claim}\label{C1}
Given a valid triangle $x + y + z = 0$, it has a probability at least $2/5$ of being good.
\end{claim}

\begin{proof}[Proof of Claim \ref{C1}]
We first show that the probability that $x$ is in another valid triangle is at most $1/5$. Indeed, for each $y'$ that forms a triangle with $x$, it has a probability of $\frac{L-1}{N-1} \leq \frac{L}{N}$ of lying in $I_Y$, because choosing values $a$, $b$ and $d$ such that $x + y + z = 0$ is valid is equivalent to choosing values $r, s \in [-l,l]$ so that $x = a + rd, y = b + sd$, and then choosing any value of $d$. Thus, each value of $y' \neq y$ will occur with probability $\frac{L-1}{N-1}$ since, for each choice of $t \neq s \in [-l,l]$, there is exactly one choice of $d$ such that $y = b + td$.

There are at most $\delta_2 N$ possible choices of $y'$ to consider, so the union bound guarantees that the probability that $x$ is in another valid triangle is at most $L\delta_2 \leq 1/5$.

The same argument applies to the probability that $y$ is in another valid triangle. For $z$, it turns out that the bound is even stronger, because for each $y'$ forming a valid triangle with $z$, $y'$ has a probability of at most $\frac{L}{N}$ of lying in $I_Y$. This is an upper bound for the probability that, setting $x' = -y'-z$, the triangle $x' + y' + z = 0$ is valid since $x'$ is not guaranteed to lie in $I_X$.

Thus, the probability that either $x$, $y$ or $z$ cause the triangle to be not good is at most $3/5$ by the union bound, and thus the probability that the triangle is good is at least $2/5$.
\end{proof}

\begin{claim}\label{C2}
Given $x \in I_X$, the probability that it is in a good triangle is at least $\frac{\delta_1}{50\delta_2}$.
\end{claim}

\begin{proof}[Proof of Claim \ref{C2}]
For each $y$ that forms a triangle with $x$, it has a probability of $\frac{L}{N}$ of lying in $I_Y$, and, conditioned on this, a probability of at least $2/5$ of forming a good triangle with $x$. In other words, for each $y$ forming a triangle with $x$, it has a probability of at least $\frac{2L}{5N} \geq \frac{1}{50\delta_2 N}$ of forming a good triangle with $x$.

By definition, $x$ can be in at most one good triangle, so these events are disjoint. There are at least $\delta_1 N$ choices of $y$ forming a triangle with $x$, and so the probability that at least one of them is good is at least $\frac{1}{50 \delta_2 N} \times \delta_1 N = \frac{\delta_1}{50\delta_2}$.
\end{proof}

\begin{claim}\label{C3}
The expected number of $x \in X$ in good triangles is at least $\frac{|X| \delta_1}{1000 \delta_2^2 N}$.
\end{claim}

\begin{proof}[Proof of Claim \ref{C3}]
The probability that $x$ is in $I_X$ is $\frac{L}{N} \geq \frac{1}{20 \delta_2 N}$.  Conditioned on this, $x$ has a probability of at least $\frac{\delta_1}{50\delta_2}$ of being in a good triangle. Hence, each $x \in X$ has a probability of $\frac{\delta_1}{1000\delta_2^2 N}$ of lying in a good triangle. Linearity of expectation yields the result.
\end{proof}

Thus, we can choose parameters $a$, $b$ and $d$ in such a way that there are at least $\frac{|X| \delta_1}{1000 \delta_2^2 N}$ good triangles in the corresponding sets $I_X$, $I_Y$ and $I_Z$.

Now, observe that we may map the intervals $I_X$, $I_Y$ and $I_Z$ into $\Z/M\Z$, where $M = \delta_2^{-1}$, in the obvious way. For example, $a + rd \in I_X$ for $r \in [-l,l]$ maps to $r \mod M$. It is clear that this map sends triangles to triangles; since $M > 8l$, this map also preserves the status of not being a triangle. 

In other words, the at least $\frac{|X| \delta_1}{1000 \delta_2^2 N}$ good triangles we found earlier correspond to an additive matching within $\Z/M\Z$. Given our hypothesis on the size of an additive matching, we deduce that 

\begin{equation}
\frac{|X| \delta_1}{1000 \delta_2^2 N} \leq Mf(M)
\end{equation}

\noindent and so

\begin{equation}
\frac{|X|}{N} \leq 1000 \frac{\delta_2}{\delta_1} f(\delta_2^{-1})
\end{equation}

\noindent which is exactly what we sought.

\end{proof}

Next, we prove an analogue of Lemma 6 from \cite{FLpaper}.

\begin{lemma}\label{SmallLem}
Suppose that $\eps, \delta > 0$ satisfy 

\begin{equation*}
\eps \gg g(\delta) \sqrt{f\left(\frac{1}{g(\delta)\delta}\right)}
\end{equation*}

\noindent for the functions $f$ and $g$ defined previously, and that $\delta < \alpha$ as defined immediately before Theorem \ref{MainTheorem}. Suppose we have a collection of $\eps N$ disjoint triangles $x_i + y_i + z_i = 0$, and let $X = \{x_i\}$, defining $Y$ and $Z$ analogously. Then there must be at least $\delta N^2$ triangles $x_i + y_j + z_k = 0$.
\end{lemma}

\begin{proof}
The majority of the proof is the same as that in  \cite{FLpaper}, so we will not reproduce it here; the only difference being that we do not mind if elements are in more than one out of $X,$ $Y$ and $Z$, because we are treating them separately in our proof of Lemma \ref{BigLem}. Suffice it to say that we will reach a point where, for some $\delta' \leq \delta$, we have at least $\frac{\delta'}{2} N^2$ triangles in sets $X$, $Y$ and $Z$, where those sets are of size at least $\frac{\eps}{2g(\delta')}N$. These have the property that each element is in at least $\frac{\delta'}{6\eps} N$ and at most $\frac{g(\delta')\delta'}{\eps} N$ triangles.

Applying Lemma \ref{BigLem}, we deduce that

\begin{align*}
\frac{\eps}{2g(\delta')} &\ll (6g(\delta')) f\left(\frac{\eps}{g(\delta')\delta'}\right) \\
\eps &\ll g(\delta')^2 \eps^{-1} f\left(\frac{1}{g(\delta')\delta'}\right) \\
\eps^2 &\ll g(\delta)^2 f\left(\frac{1}{g(\delta)\delta}\right)
\end{align*}

\noindent where in the second line we used the conditions on $f$ and in the third line we used that the right hand side decreases as $\delta'$ decreases.

\end{proof}

We are now ready to prove Theorem \ref{MainTheorem}.

\begin{proof}[Proof of Theorem \ref{MainTheorem}]
We follow the same strategy as in \cite{FLpaper}. Suppose that $A$, $B$ and $C$ are such that we cannot remove all the triangles without removing at least $\eps N$ elements from $A$, $B$ and $C$. Any maximal set of disjoint triangles must have size at least $\frac{\eps}{3} N$, else we could remove all of the elements of those triangles and there would be no triangles left.

Lemma \ref{SmallLem} guarantees that we must have at least $\delta N^2$ triangles in total, where 

\begin{equation*}
\eps \ll g(\delta) \sqrt{f\left(\frac{1}{g(\delta)\delta}\right)}
\end{equation*}

\noindent as required.
\end{proof}

\section{Theorem \ref{MainTheorem} for general $N$}

In this section, we complete the proof of Theorem \ref{MainTheorem} in the case that $N$ need not be prime. First, observe that the prime case of Theorem \ref{MainTheorem} implies that it holds for subsets of $[-M/2,M/2]$. Indeed, provided that the functions $f$ and $g$ exist and satisfy all of the requirements imposed upon them above, then we can deduce the following:

\begin{corollary}\label{CorZed}
Suppose that $A$, $B$ and $C$ are subsets of $[-M/2,M/2]$ with the property that there are at most $\delta M^2$ triangles $a + b + c = 0$ with $a \in A$, $b \in B$ and $c \in C$.

Then, we can remove all of the triangles by deleting at most $\eps M$ elements from $A$, $B$ and $C$, where $\eps$ satisfies

\begin{equation}
\eps \ll g(\delta) \sqrt{f\left(\frac{1}{g(\delta)\delta}\right)}.
\end{equation}

\end{corollary}

\begin{proof}
Suppose we have sets $A$, $B$ and $C$ which define $\delta M^2$ triangles. Select a prime $N$ such that $2M \leq N \leq 4M$, and consider the reduction modulo $N$ map $\phi$ taking $[-M/2,M/2]$ to $\group$. This preserves the status of being a triangle, as well as the status of not being a triangle.

The image of $(A,B,C)$ under $\phi$ contains at most $\delta (N/2)^2$ triangles, and thus requires removal of at most $\eps N$ points to remove all of the triangles, where $\eps$ satisfies

\begin{equation}
\eps \ll g(\delta/4) \sqrt{f\left(\frac{4}{g(\delta/4)\delta}\right)}.
\end{equation}

Thus, to remove the triangles from $(A,B,C)$, the deletion of at most $\eps N \leq 4\eps M$ points is necessary. By adjusting the implicit constant, we deduce Corollary \ref{CorZed}.
\end{proof}

We may now deduce that Theorem $\ref{MainTheorem}$ holds in arbitrary finite cyclic groups:

\begin{corollary}\label{CorComp}
Theorem \ref{MainTheorem} holds without the requirement that $N$ is prime.
\end{corollary}

\begin{proof}

Suppose not; then for some (composite) $N$, $\group$ contains sets $A$, $B$ and $C$ which define at most $\delta N^2$ triangles, but require deletion of at least $\eps N$ points to remove the triangles, and where $\eps$ does not satisfy (\ref{MTEQ}) (with a slightly adjusted implicit constant).

As in the proof of Theorem \ref{MainTheorem}, a greedy argument guarantees the existence of $\frac{\eps}{3} N$ disjoint triangles. Consider the map $\pi$ taking $\group$ to $[0,N-1]$ in the obvious way; a triple $(a,b,c)$ in $\group$ is a triangle if and only if its image under $\pi$ has sum either $N$ or $2N$.

For one of the two possibilities for the sum, there are at least $\frac{\eps}{6} N$ disjoint triangles. In the first case, in which the triangles have sum $N$ in the image of $\pi$, consider $\pi(A)$, $\pi(B)$ and $\pi(C) - N$, and in the second case, consider $\pi(A)$, $\pi(B) - N$ and $\pi(C) - N$. Either way, we have at most $\frac{\delta}{2} (2N)$ triangles in $[-N,N]$, which require deletion of at least $\frac{\eps}{12} (2N)$ points to remove, because they define at least that many disjoint triangles. Corollary \ref{CorZed} gives us the result.
\end{proof}

\clearpage

\bibliographystyle{plain}                                
\bibliography{RemovalRefs}

\end{document}